\documentclass[11pt,reqno]{amsart}
 \usepackage[dvips]{epsfig}
 \usepackage{amsgen, amstext,amsbsy,amsopn, amsthm, amsfonts,amssymb,amscd,amsmat
 h,euscript,enumerate,url,verbatim,calc,xypic}

 \usepackage{latexsym}
 \usepackage{graphics}
 \usepackage{color}

 \newcommand{\n}{\mathfrak{n} }
 \newcommand{\m}{\mathfrak{m} }
 \newcommand{\q}{\mathfrak{q} }

  \newcommand{\N}{\mathbb{N}}

 \newcommand{\ov}{\overline}

 \newcommand{\dm}{\operatorname{dim}}
 \newcommand{\h}{\operatorname{ht}}

 \newcommand{\gr}{\operatorname{gr}}

\theoremstyle{plain}
 \newtheorem{theorem}{Theorem}[section]
 
 \newtheorem{corollary}[theorem]{Corollary}

 \theoremstyle{definition}
 
 \newtheorem{remark}[theorem]{Remark}

 \newtheorem{example}[theorem]{Example}
 \theoremstyle{remark}

 \title[ ] {The Tangent cone of a local ring of  codimension 2}

 \author{Mousumi Mandal and Maria Evelina Rossi}
 \thanks{The first author was supported by INdAM-COFUND Marie-Curie Fellowship.   The second  author  was  supported  by MIUR, PRIN 2010-11 (GVA).  This work was partly accomplished while the first author  was visiting the University of Genoa.}

 \address{Department of Mathematics, Indian Institute of Science Bangalore, 5600012, India} \email{mousumi@math.iisc.ernet.in}
 
 \address{Dipartimento Di Matematica, Universita' Di Genova, Via Dodecaneso 35, 16146, Genova, Italy} \email{rossim@dima.unige.it}

\begin{document}
\maketitle
\begin{center}{\it Dedicated to Professor Ngo Viet Trung  on the  occasion of his 60th birthday}
\end{center}

\begin{abstract}
Let  $(S, \n)  $ be  a   regular local ring and let $I \subseteq \n^2 $ be a perfect  ideal of $S. $ Sharp upper bounds on the minimal number of generators of $I$ are known  in terms of the Hilbert function of  $R=S/I. $ Starting from information on  the ideal $I,  $ for instance the minimal number of generators, a   difficult task  is to determine  good  bounds on the minimal number of generators of  the leading ideal $I^*  $ which defines the tangent cone of $R$ or  to give information on its  graded structure.          Motivated by papers of  S.C. Kothari, S. Goto {\it{et al.}} concerning the leading ideal  of a complete intersection $I=(f,g) $  in a regular local ring,     we present results provided ht$(I)=2.$     If $I$ is a complete intersection,  we prove that the Hilbert function of $R$  determines  the graded Betti numbers  of the leading ideal and, as a consequence,   we recover  most of the results of the previously  quoted papers.  The description is more complicated if $\nu(I) >2$ and a careful investigation  can be  provided when $\nu(I)=3. $  Several examples illustrating  our results are given.

\end{abstract}
 
\section*{Introduction and Notation}
 
Let $(R, \m, k)  $ be a local ring with maximal ideal $\m$ and residue field $k.$  
The associated graded ring  $G=\gr_{\m}(R)=\oplus_{i\geq 0}\m^i/\m^{i+1}$ corresponds to a relevant geometric construction in the case $R $ is the localization at the origin O of the coordinate ring of an affine variety $V$ passing through O. It turns out that  $G $ is the coordinate ring of the {\it Tangent Cone} of $V$ at O, which is the cone composed of all lines that are the limiting positions of secant lines to $V$ in O.  Consider a  minimal Cohen presentation $S/I $ of $R$ where $(S, \n, k) $    is  a   regular local ring and   $I \subseteq \n^2 $ is   an ideal of $S.  $ We recall that $G \simeq P/I^* $ where $P=\gr_{\n}(S) $ is the polynomial ring and $I^*$ is the homogeneous ideal of $P$ generated by the initial forms of the elements of $I.$ From the algebraic point of view, the local ring $R$ and the  standard graded $k$-algebra $G$ share the same Hilbert function. In fact, by definition, the Hilbert function  of  $R$ is  the numerical function $HF: \N \to \N $ such that $HF_R(i)=\dim_k \m^i/\m^{i+1}.$  

\vskip 2mm
We address our interest  to the  structure of $I^*, $ called the leading ideal of $I, $    which encodes   the algebraic and the geometric information on $I.$  
Denote by $\nu(I)$ the minimal number of generators of $I, $ then by a classical result of Krull and by the definition of $I^*, $ it is  known that

\begin{equation}  \label{leq} \h(I) \le \nu(I) \le \nu(I^*). \end{equation}
Upper bounds for $\nu(I^*)$ (and hence for $\nu(I)$) are known in terms of the multiplicity of $R$ or,  more precisely, in terms  of the Hilbert function of $R,  $  see for instance \cite{BrIar}, \cite{Elias}, \cite{ERV}, \cite{Sally1},  \cite{Sally2}.   

\vskip 2mm
Starting from information on the ideal $I,  $   a  more difficult task  is to prove sharp lower bounds for $I^*  $  improving   $  \nu(I) \le \nu(I^*)$, since this  involves the study of  the structure of $I^*.  $  
 
The work  of Goto {\it{et al.}}   (see \cite{gkh}, \cite{gkm1}, \cite{gkm2})  evidences  the difficulty of  the problem,   even if we assume that $I=(f,g) $ is generated by a regular sequence in a $2$-dimensional regular local ring $S.$  

\vskip 2mm
More in general, if not otherwise specified,    we assume $I \subseteq \n^2 $ is an ideal of  height  $2 $ (not necessarily a complete intersection)  in a $r$-dimensional  regular local ring $(S, \n)  $   such that $G$ is Cohen-Macaulay. Under our assumption $I$ and $I^*$ are homogeneous perfect ideals of codimension two, hence good  information  come from the Hilbert-Burch theorem. The main results of the paper are contained in Corollary  \ref{bound2}, Theorem \ref{unique}, Theorem \ref{integer} and Theorem \ref{three}.  
\vskip 2mm

Let $R=S/I$ and denote by $HS_R(t) = \sum_{j\ge 0} HF_R(j) t^j $ the Hilbert series of $R.$ It is well known that  
$$HS_R(t) = \frac{h_R(t)}{(1-t)^{r-2}} $$ where $h_R(t) = 1 +h_1 t + \dots + h_s t^s   \in  \mathbb Z[t] $ is called $h$-polynomial   and $h_R=(1, h_1, \dots, h_s)$ is called $h$-vector. If $G$ is Cohen-Macaulay of codimension two, 
 then $h_R(t) $ represents the Hilbert series of any Artinian reduction of $R$ (we may assume that $k$ is infinite), hence,   by Macaulay's inequalities, the $h$-vector  
\begin{equation} \label{H}  h_R=(1,2,\ldots ,d, h_d,\ldots ,h_s)\  \mbox{   verifies } \   d\geq h_d\geq h_{d+1}\geq \cdots \geq h_s\geq 1  \end{equation} where  $d$ is  the order of  $I, $ that is the maximum integer such that $I \subseteq \n^d. $

\vskip 2mm From now on,  a numerical function $h$ verifying (\ref{H})    is  called {\it{$O$-sequence}}.  Motivated by the above observation, we assume $R$ is an Artinian local ring of embedding dimension two.  Macaulay proved that, given an $O$-sequence  $h,$ then  there exists an unique lex-segment ideal $L$ in $P$  such that $h=HF_{P/L}. $   Then   using this fact and the Hilbert-Burch Theorem, it is known that (\ref{leq}) specializes to 
\begin{equation}  \label{2} \nu(I) \le \nu(I^*) \le  \nu(L)=d+1.
\end{equation}

 \vskip 2mm
 
The Hilbert function of an Artinian local ring $R$ of embedding dimension $2$ has been studied by several  authors.  

Iarrobino in \cite{I} proved that if $p: = \max\{ |HF_R(i)-HF_R(i+1)| \mid i=1,\ldots,s\}, $ then 
\begin{equation} \label{p}   \nu(I) \ge p+1. \end{equation}
If $ h $ is an $O$-sequence  and $m$ is any  integer verifying $ p+1 \le m \le d+1,$   then Bertella in \cite{b} gave   an effective method  for constructing  an ideal $I\subseteq S=k[|x,y|]$ such that   $\nu(I)=m$ and $HF_{S/I}=h $ (see also \cite[Remark 4.7]{rs}).

In general  $\nu(I^*) \geq m=\nu(I)  $ and  Theorem \ref{bound}  gives a more precise bound.  Notice that the ideals $I$ and $I^*$ share the same Hilbert function, but   in general a different  minimal number of generators occurs. For instance, Example \ref{ex1} shows that  there exist ideals $I$ whose  Hilbert function forces  $\nu(I^*) \ge  \nu(I)+3. $   

\vskip 2mm

\vskip 2mm
As we will explain later, an important tool in our approach is  the technique   of  the consecutive  cancellations described by Peeva in \cite{p} in the graded case and,   by Rossi and Sharifan,  in   \cite{rs} in the local case. It is known that  in codimension two each cancellation can be realized (see \cite{b}). 
 
\vskip 2mm
 Notice that  if $m=2, $ hence  $I=(f,g) $ is generated by a regular sequence,  we prove that the Hilbert function   fixes  the graded Betti numbers of $I^*$ (see Theorem \ref{unique}).  In particular we prove that there is a one to one correspondence between the $O$-sequences  $h$  such that $h=HF_{S/I}$ with  $I=(f,g)$  and  two sequences of integers $(c_1, \dots, c_n) $ and $(e_1, \dots, e_n) $   related by suitable conditions described in Theorem \ref{multiplicity}. The two sequences $(c_1, \dots, c_n) $ and $(e_1, \dots, e_n) $ give also necessary conditions for constructing a complete intersection whose leading ideal has a Betti table determined by these invariants, see Theorem \ref{integer}. 

The  results recover most of the results proved by S. Goto, W. Heinzer and M.K. Kim   in \cite{gkm1} and \cite{gkm2}  where many technical computations were necessary.  In the quoted papers the authors  sharpened a result  by Kothari  in \cite{k} who answered a question  raised by Abhyankar concerning the  Hilbert function of the localization of a pair of plain curves. The new approach  allows  us easier proofs  because it takes advantage of the numerical invariants  coming from a minimal free resolution of $I^*.$ 
 \vskip 2mm

 \vskip 2mm

 If $\nu(I) >2, $ then the  Hilbert function of $R$ does no longer fix the Betti numbers of $I^*$ and   a careful investigation is done if $\nu(I) =3. $

\vskip 2mm

\section{Lower bound for $\nu(I^*)$}
 Given a numerical function $h: \N \to \N, $ let us denote by $ \triangle h(j)=h(j)-h(j-1)$ the first difference and by $ \triangle^2h(j)=\triangle h(j)-\triangle h(j-1)=h(j)-2h(j-1)+h(j-2)$ the second difference operator. Denote by $P$ the polynomial ring $k[x_1,\dots, x_n] $ in $n$ indeterminates.
 
\vskip 2mm
Since $I^*$ is a  homogeneous perfect  ideal of codimension two in the polynomial ring $P,  $ for sake of completeness we insert here some general facts  that  will be useful in the present work. 

 Let $J $ be a perfect homogeneous  ideal  of height   two in $P$  with minimal 
 $P$-free graded   resolution:
 $$ {\bf{F.}}: \  0    \to \oplus_{j \ge 0 }  P^{\beta_{1 j}} (-j)    \overset{\phi} \to  \oplus_{j \ge 0 }  P^{\beta_{0j}} (-j)     \to J  \to 0.    $$   
By the Hilbert-Burch Theorem,  $J$ can be generated by the maximal minors of an homogeneous matrix associated to $\phi $ of size $m \times (m-1)$ where $m= \sum_j  \beta_{0j}. $  
 
 \vskip 2mm
 Accordingly with \cite{rs}, we recall  that    ${\bf{F.}} $ 
 admits a  {\it negative   cancellation }  (resp.   {\it   zero    cancellation })  if there exist integers   $  j < j'  $ (resp. $j=j' $) such that  $\beta_{1 j}, \beta_{0 j'}>0.$ We will denote it by $(P(-j), P(-j')).$ 
 \vskip 2mm
 For example   $ 0  \to P (-3) \oplus P(-5)\oplus P(-6)  \to P^2(-2) \oplus P^2(-5)  \to \dots   $ admits a zero cancellation ($P(-5), P(-5)$) and a negative cancellation $(P(-3), P(-5))$.
 \vskip 2mm

  The following result gives a lower bound   for the number of generators of a  homogeneous perfect  ideal $J$ of codimension two in the polynomial ring $P $   in terms of the second difference operator of its $h$-vector. 
 In particular we prove that when $J$ is minimally generated with respect a bound given by its  Hilbert function, then the graded Betti numbers of $J$ are uniquely determined.

\vskip 2mm
 
   We denote by $|n| $ the positive value of the integer $n.$

\begin{theorem}\label{bound}
Let $h:=(1,2,\ldots, d,h_d,\ldots ,h_s)$ be an $O$-sequence and  define the sets $\mathcal I:=\{j|\triangle^2h(j)\leq -1\}$ and $\mathcal J:=\{j|\triangle^2h(j)\geq 1\}$. Let $J$ be a perfect homogeneous ideal of height two of $P$ and such that $h $ is  the $h$-vector of $G=P/J.$  Then

 \vskip 2mm

\begin{enumerate}
\item  $\displaystyle{ \nu(J) \geq \sum_{i\in \mathcal I}|\triangle^2h(i) |}$. 
\item   If the equality holds,  then the Betti numbers of $J$ are uniquely determined by the Hilbert function  and  
\[
\beta_{0j}(J) = \left\{
 \begin{array}{ll}
  |\triangle^2h(j) | &\mbox{ for }j\in \mathcal I\\
  0 &\mbox{ otherwise ,}
      \end{array}
      \right.
\]
\[
      \beta_{1j}(J) = \left\{
       \begin{array}{ll}
        |\triangle^2h(j) | &\mbox{ for }j\in \mathcal J\\
        0 &\mbox{ otherwise. }
            \end{array}
            \right.
             \]

\end{enumerate}

\end{theorem}

\begin{proof}
 
We may assume that the residue field $k$ is infinite.  Since $G$ is Cohen-Macaulay of dimension $n-2,$ let $Q=(\ell_1,\ldots ,\ell_{n-2})$ be an  ideal of $P$  generated by linear forms such that $\ov Q = Q+J/J $ is generated by a maximal $G$-regular sequence. Denote $\ov G= G/\ov Q $ and $\ov J = J+Q/Q.$   Now the following facts hold:

\vskip 2mm
- $\beta_{ij}^P  (G) = \beta_{ij}^{P/Q} (\ov G).$

\vskip 2mm
- $HS_G(t) =\frac{HS_{\ov G} (t)}{(1-t)^{n-2}}.$
\vskip 2mm

In particular $HS_{\ov G} (t) = h$ and $\nu(J) =\nu(\ov J).$ Hence from now on we assume that $J \subseteq P= k[x,y] $ and $G=P/J $ is Artinian with Hilbert function determined by $h.$ 

 Let $L$ be the lex-segment ideal associated to $h$. Then $\nu (L)=d+1$. Let us denote by $p_j=\triangle h(j) $ and $q_j=\triangle^2h(j)$.   Since the lex-segment ideal $L$ is Borel fixed,  by Eliahou-Kervaire's  resolution (see \cite{EK}),   the graded minimal free resolution of $L$ is the following  
$$0\longrightarrow \bigoplus_{j=1}^t P^{|p_{c_j}|}(-c_j-1)\longrightarrow P(-d)\bigoplus_{j=1}^tP^{|p_{c_j}|}(-c_j)\longrightarrow L\longrightarrow 0,$$
where $d\leq c_1<c_2< \cdots < c_t, $  $ \beta_0(L)=\sum_{j=1}^t|p_{c_j}|=d+1 $ and   $|p_{c_j}| $ denotes  the number of generators of $ L $ of degree $c_j. $    By \cite[Theorem 1.1]{p} the minimal free resolution of $J $ is obtained from the minimal free resolution of $L$ by a sequence of consecutive zero cancellations. Therefore $\nu(J )=d+1-r$, where $r$ denotes the number of zero cancellations. Hence it is enough to compute the maximum number of zero cancellations that we may perform. 

 In the minimal free resolution of $L$, $P(-d)$ and $P(-c_1)$ cannot  be cancelled since $c_i+1>c_1\geq d $ for  every $ i=1,\ldots ,t$. Now we have to compute $\beta_{0c_j}(J )$ for $j\geq 2$. The zero cancellation at  $c_j$-position with multiplicity $\alpha$ will be denoted by $(P^{\alpha}(-c{_j}),P^{\alpha}(-c_j))$.  A resolution admits a zero cancellation at  $c_j$-th position only if $c_j=c_{j-1}+1$. Let us assume that we have the zero cancellation $(P^{|p_{c_{j-1}}|}(-c_j), P^{|p_{c_j}|}(-c_j))$ at the $c_j$-th position.
If $q_{c_{j}}=p_{c_{j}}-p_{c_{j-1}}<0$ then, after  the   above cancellation,    we obtain $(0,P^{|q_{c_{j}}|}(-c_j))$ and hence  $\beta_{0c_j}(J )=|q_{c_j}|$ and $\beta_{1c_j}(J)=0$. If $q_{c_j}>0$ then the zero cancellation at $c_j$-th place can be replaced by $(P^{|q_{c_j}|}(-c_j),0)$ having  $\beta_{0c_j}(J )=0$ and $\beta_{1c_j}(J)=|q_{c_j}|$. If $q_{c_j}=0$ then it will be replaced by $(0,0)$ and in this case $\beta_{0c_j}(J)=0$ and $\beta_{1c_j}(J)=0$. In any case if we perform the zero cancellation at the $c_j$-th place then $\triangle^2h$ determines $\beta_{0c_j}(J)$ and $\beta_{1c_j}(J)$. 
Assume now $c_{j-1}+2\leq c_j\leq c_{j+1}-2, $   since  we cannot perform zero cancellation at $c_j$ and $c_{j+1}$-th positions we have $p_{c_j-1}=0$ and $p_{c_j+1}=0$ and hence $\beta_{0c_j}(J)=|p_{c_j}|=|p_{c_j}-p_{c_j-1}|=|q_{c_j}|$. In particular $q_{c_j}<0$. Moreover, $\beta_{1c_j+1}(J)=|p_{c_j}|=|p_{c_j+1}-p_{c_j}|=|q_{c_j+1}|$. In particular $q_{c_j+1}>0.$ Thus if we perform all the possible consecutive zero cancellations in the minimal free resolution of $L$ then $\triangle^2h$ determine the graded Betti numbers of the minimal free resolution of $J$. 
Therefore we have $\nu(J)=\sum_{j\geq 1}\beta_{0j}(J)\geq \sum_{i\in \mathcal I}|\triangle^2h(i)|$.
Now if equality holds or equivalently  if we perform all the possible consecutive zero cancellations in the minimal free resolution of $L$ then the minimal free resolution of $J $ has the following form
$$0\longrightarrow\bigoplus_{i\in \mathcal J}P^{\triangle^2h(i)}(-i) \longrightarrow \bigoplus_{i\in \mathcal I}P^{|\triangle^2h(i)|}(-i)\longrightarrow J \longrightarrow 0,$$
which implies that the graded Betti numbers of $J$ are uniquely determined by $h$ as follows
\[
\beta_{0j}(J) = \left\{
 \begin{array}{ll}
  |\triangle^2h(j) | &\mbox{ for }j\in \mathcal I\\
  0 &\mbox{ otherwise, }
      \end{array}
      \right.
\]
\[
      \beta_{1j}(J) = \left\{
       \begin{array}{ll}
        |\triangle^2h(j) | &\mbox{ for }j\in \mathcal J\\
        0 &\mbox{ otherwise. }
            \end{array}
            \right.
             \]
 \end{proof}

We apply the above result in the following situation of our major interest.
\begin{corollary} \label{bound2}  Let $h:=(1,2,\ldots, d,h_d,\ldots ,h_s)$ be an $O$-sequence and  define the sets $\mathcal I:=\{j|\triangle^2h(j)\leq -1\}$ and $\mathcal J:=\{j|\triangle^2h(j)\geq 1\}$. Let $I \subseteq  \n^2$ an  ideal of regular local ring $(S,\n) $   such that $h $ is the $h$-vector of $R=S/I.$  Assume $G=gr_{\m}(R) $ is Cohen-Macaulay, then
 $${ \nu(I^*) \geq \sum_{i\in \mathcal I}|\triangle^2h(i) |}.$$ 
  If the equality holds,  then  
  \[
\beta_{0j}(I^*) = \left\{
 \begin{array}{ll}
  |\triangle^2h(j) | &\mbox{ for }j\in \mathcal I\\
  0 &\mbox{ otherwise, }
      \end{array}
      \right.
\]
\[
      \beta_{1j}(I^*) = \left\{
       \begin{array}{ll}
        |\triangle^2h(j) | &\mbox{ for }j\in \mathcal J\\
        0 &\mbox{ otherwise. }
            \end{array}
            \right.
             \]

\end{corollary}

\begin{remark} By the  effective method proved by Bertella in \cite{b}, it is always possible to find an ideal  attaining the minimal value according to Corollary \ref{bound2}. Let $h:=(1,2,\ldots, d,h_d,\ldots ,h_s)$ be an $O$-sequence. 

Define  $\mathcal I:=\{j|\triangle^2h(j)\leq -1\}$  and    $p: = \max\{|\triangle h(i)|  \mid i=1,\ldots,s\}.   $ Then there exists an ideal 
 $I $  of $S= k[[x,y]] $   such that 

 \begin{enumerate}
 \item $h=HF_{S/I}  $, 
\item  $  \nu(I^*) = \sum_{i\in \mathcal I}|\triangle^2h(i) | $, and
\item  $\nu(I) =m $ for every integer $m$ such that $ p+1 \le m \le  \nu(I^*).$
\end{enumerate}

\end{remark}

The following example will be useful to clarify the above remark.

\begin{example} \label{ex1}
Consider the following $O$-sequence:
$$ h =(1, 2, 3, 4, 5, 6, 7, 8, 9, 10, 10, 10, 9, 8, 8, 5, 3, 3, 2, 0, 0).$$
Then
 $\triangle h= (1, 1, 1, 1, 1, 1, 1, 1, 1, 1, 0, 0,  -1, -1, 0, -3,  -2, 0,  -1, -2, 0 ) $ and 
 $ \triangle^2 h=$ 
 
 $= (1,0, 0, 0, 0, 0, 0, 0, 0, 0, 0, -1, 0, -1, 0, 1, -3, 1, 2, -1, -1, 2 ). $ 

Our goal is to construct an ideal   $I \subseteq S=k[[x,y]] $  such that $h =HF_{S/I} $ and $  \nu(I^*) = \sum_{i\in \mathcal I}|\triangle^2h(i) | $. Then  the order  of   $I$ is $d=10$ and the maximal jump in the Hilbert function   is $p=\max_{i\geq 1}\{|\triangle h(i)|\}=3$.   Moreover,  by  Equations (\ref{2}) and (\ref{p}),  we have  $4\leq \nu(I) \leq 11 $  and, by Theorem \ref{bound},  we deduce $7  \leq \nu(I^*) \leq 11.  $ 
Bertella's construction suggests to realize $I^*$ by a  suitable deformation  of the Hilbert-Burch matrix associated to the (unique) lex-segment ideal $L$ such that $h = HF_{S/L}.$ On his turn, the ideal $I$ will be obtained in the analogous way by the matrix associated to $I^*.$     

\vskip 2mm
It is easy to verify that  $$L=(x^{10},x^9y^3,x^8y^5,x^7y^8,x^6y^9,x^5y^{10},x^4y^{12},x^3y^{13},x^2y^{16},xy^{18}, y^{19})$$ is the lex-segment ideal associated  to $h$. The minimal free resolution of $L$ is as follows:
{\small{$$0\longrightarrow P(-13)\oplus P(-14)\oplus P^3(-16)\oplus P^2(-17)\oplus P(-19)\oplus P^2(-20)\longrightarrow$$ $$P(-10)\oplus P(-12)\oplus P(-13)\oplus P^3(-15)\oplus P^2(-16)\oplus P(-18)\oplus P^2(-19)\longrightarrow L\longrightarrow 0$$}}
The above resolution admits   at most 4  consecutive zero cancellations:  $(P(-13),P(-13))$, $(P^2(-16),P^2(-16))$ and $(P(-19),P(-19))$. Hence, by Theorem \ref{bound},  $\nu(I^*) \geq 7. $ For having $\nu(I^*)=7  $   we have to perform all the above  cancellations. Then we obtain the minimal free resolution of $I^*$ as follows
{\small{$$0 \to P(-14)\oplus P(-16)\oplus P^2(-17)\oplus P^2(-20) \longrightarrow$$ $$
 P(-10)\oplus P(-12)\oplus P^3(-15)\oplus P(-18)\oplus P(-19)\to I^* \longrightarrow  0$$}}with $\beta_{0j}(I^*)=|\triangle^2h(j)|$ for $j\in \mathcal I$ and $\beta_{1j}(I^*)=|\triangle^2h(j)|$ for $j\in \mathcal J$.  
Consider the Hilbert-Burch matrix associated to $L,  $ we may realize an ideal $I^*$ having the above resolution replaicing  $0$ by $1$ in the entries corresponding to the zero cancellations.     Then $I^*$ 
is  the ideal generated  by the maximal minors of the following matrix  
  \[M:=
 \begin{pmatrix}
 y^3& 0& 0 &0 &0 &0 &0 &0& 0& 0\\
 -x& y^2& 0& 0& 0& 0& 0& 0& 0& 0\\
 1& -x& y^3& 0& 0& 0& 0& 0& 0& 0\\
 0& 0& -x& y& 0& 0& 0& 0& 0& 0\\
 0& 0& 0& -x& y& 0& 0& 0& 0& 0\\
 0&0& 0& 0& -x& y^2& 0& 0& 0& 0\\
 0&0& 1& 0& 0& -x& y& 0& 0& 0\\
 0& 0& 0& 1& 0& 0& -x& y^3& 0& 0\\
 0& 0& 0& 0& 0& 0& 0& -x& y^2& 0\\
 0& 0& 0& 0& 0& 0& 0& 1& -x& y\\
 0& 0& 0& 0& 0& 0& 0& 0& 0& -x
 \end{pmatrix}
 \]

In particular $ I^*=(x^{10} - 2x^8y^2 - x^6y^4 + 4x^4y^6 - 2x^2y^8, -x^9y^3 + x^7y^5 + 2x^5y^7 - 2x^3y^9, -x^7y^8 + x^5y^{10} + x^3y^{12} - xy^{14}, x^6y^9 - x^4y^{11}, -x^5y^{10} + x^3y^{12}, x^2y^{16}, y^{19}).  $   
 
The resolution of $I^*$ admits the following three negative cancellations $(P(-14), P(-15))$, $(P(-16), P(-18))$ and $(P(-17), P(-19))$. Hence $4 \leq m=\nu(I) \leq 7.$ For having an ideal with $\nu(I) =m, $ by Rossi and Sharifan's result in \cite[Remark 4.7]{rs} we have to perform $7-m $  negative cancellations. As before, by replacing in the matrix $M$ the entries $0$'s by $1$'s in the positions corresponding to the negative cancellations,   we obtain the wanted  ideal.  For instance, an ideal $4$-generated whose    leading ideal is $I^*$ :

  $I=(x^{10} - 2x^8y^2 - x^6y^4 + 4x^4y^6 - 2x^2y^8+x^4y^9 + x^3y^{10} - x^6y^6 - x^5y^7 + x^4y^8 + x^3y^9 - x^8y^3 + x^6y^5 + x^4y^7 - x^2y^9,  - x^9y^3 + x^7y^5 + 2x^5y^7 - 2x^3y^9 -x^3y^{12 }- x^2y^{13} + x^5y^9 + x^4y^{10} + x^7y^6 - x^5y^8 - x^3y^{10} + xy^{12},  x^6y^9 - x^4y^{11} + y^{17} - x^5y^{10} + x^3y^{12}, -x^2y^{15 }- xy^{16},    -x^5y^{10} + x^3y^{12}+ y^{17}).$

 \end{example}

\section{When $I$ is a complete intersection}
In this section we investigate the structure of $I^*$ when $\nu(I)=2, $ that is,  $I$ is a complete intersection in a regular local ring $S $ of dimension two. The results can be extended to higher dimension, provided $G$ is Cohen-Macaulay. 

We prove that the numerical invariants of the graded minimal free resolution of $I^*$ are uniquely determined by the Hilbert function. This is quite unexpected because,  in general,  different Betti tables correspond to the same Hilbert function.  

We recall that if $h$ is an $O$-sequence,  then there exists a complete intersection ideal $I=(f,g) \subseteq S=k[[x,y]] $ such that $HF_{S/I} =h $ if and only if  $  |\triangle h(i)|  \leq 1  $  for every $i. $

\begin{theorem}\label{unique}
Let $h:=(1,2,\ldots, d,h_d,\ldots ,h_s)$ be an $O$-sequence and  define the sets $\mathcal I:=\{j|\triangle^2h(j)\leq -1\}$ and $\mathcal J:=\{j|\triangle^2h(j)\geq 1\}$.   Then for all ideals $I$ generated by a regular sequence in a $2$-dimensional regular local ring $S$ such that  $h=HF_{S/I}, $     the graded Betti numbers of $I^*$ are uniquely determined by $h$ as follows:
\[
\beta_{0j}(I^*) = \left\{
 \begin{array}{ll}
  |\triangle^2h(j) | &\mbox{ for }j\in \mathcal I\\
  0 &\mbox{ otherwise, }
      \end{array}
      \right.
\]
\[
      \beta_{1j}(I^*) = \left\{
       \begin{array}{ll}
        |\triangle^2h(j) | &\mbox{ for }j\in \mathcal J\\
        0 &\mbox{ otherwise. }
            \end{array}
            \right.
             \]
\end{theorem}

\begin{proof}
Since $\nu(I)=2, $ then by \cite[Theorem 2.3]{b}  we have $p=\max_{i\geq 1}\{|\triangle h(i)|\}=1.$ Let $L$ be the lex-segment ideal corresponding to $h$. By \cite[Theorem 4.1]{rs} the Betti numbers of $I$ come  from the Betti numbers of $L$ by a sequence of consecutive zero and negative cancellations.  One can prove easily  that when $I$ is a complete intersection,  then there exists a unique diagram of cancellations which forces  the Betti numbers of $I^*.$     Then by Eliahou-Kervaire's  resolution in \cite{EK} the minimal free resolution of $L$ is of the form
$$0\longrightarrow \bigoplus_{j=1}^d P(-c_j-1)\longrightarrow P(-d)\bigoplus_{j=1}^dP(-c_j)\longrightarrow L\longrightarrow 0$$
where $d\leq c_1<c_2<\cdots <c_d.$
Note that $P(-d)$ and $P(-c_1)$ can not be cancelled as $d\leq c_1<c_i+1$ for $i=1,\ldots ,d.$  Hence for having $\nu(I)=2, $   all the remaining part   $\bigoplus_{j\ge 2} ^dP(-c_j)$ must be cancelled by zero or negative cancellations.    Now  $P(-c_2)$ can be only cancelled with $P(-c_1-1)$ (zero or negative cancellation). Going on inductively on $i, $ we observe that for each $c_i, $   we have only one candidate:  $(P(-c_{i-1}-1),P(-c_i))$ and the cancellation must be performed for reaching  $\nu(I)=2.$  In particular we have to perform all the zero cancellations in the   minimal free resolution of $L$. Therefore $\nu(I^*)=\sum_{i\in \mathcal I}|{\triangle^2h(i)|}$ and then by Theorem \ref{bound}$(2)$ the Betti numbers of $I^*$ are uniquely determined by $h$ as 
\[
\beta_{0j}(I^*) = \left\{
 \begin{array}{ll}
  |\triangle^2h(j) | &\mbox{ for }j\in \mathcal I\\
  0 &\mbox{ otherwise, }
      \end{array}
      \right.
\]
\[
      \beta_{1j}(I^*) = \left\{
       \begin{array}{ll}
        |\triangle^2h(j) | &\mbox{ for }j\in \mathcal J\\
        0 &\mbox{ otherwise. }
            \end{array}
            \right.
             \]

\end{proof}

\medskip

On the spirit of the results of Goto-Heinzer-Kim proved in \cite{gkm1},  we present a deeper investigation of the structure of the leading ideal of a complete intersection by using the techniques of the present paper.

Let $(S,\n)$ be a $2$-dimensional regular local ring and $I=(f,g)$ be a complete intersection ideal with $\n$-adic valuations $v_{\n}(f)=a\leq v_{\n}(g)=b$ and $f^* \nmid g^*$. In \cite{gkm1} Goto, Heinzer and Kim have proved that if $\nu(I^*)=n, $ then $I^*$ contains a homogeneous system of generators $\{\zeta_i\}$ such that $f^*=\zeta_1, g^*=\zeta_2$ and $\deg \zeta_i+2\leq \deg \zeta_{i+1}$ for $2\leq i\leq n-1$ and they described the Hilbert series of $G$ in terms of  $c_i=$ the degree of  $\zeta_i$ and the integers $d_i= $  the degrees of $ GCD(\zeta_1,\ldots ,\zeta_i)$.

Given a  complete intersection ideal $I=(f,g), $ they  associate a positive integer $n$ with $2\leq n\leq a+1,$  an ascending sequence of positive integers $(c_1=a,\ldots ,c_n)$ and a descending sequence of integers $(d_1=c_1,d_2,\ldots ,d_n=0)$ such that $c_{i+1}-c_i>d_{i-1}-d_i>0$ for each $i$ with $2\leq i\leq n-1$. In \cite[Theorem 2.3]{gkm2} Goto, Heinzer and Kim proved that     these two sequence of integers $c_i$'s and $d_i$'s give also  sufficient conditions for the existence of a complete intersection ideal $I$ whose leading ideal has these invariants. 

\vskip 2mm
      In the following theorem we prove the analogous  of  \cite[Theorem 1.2, 1.3]{gkm1}, but we replace the sequence of $d_i$'s with a sequence, say $e_i$'s,     directly related to the minimal free resolution of $I^*$. Obviously explicit relations among these integers exist  and they  will be  discussed in Remark \ref{relation}. The new approach  allows  us easier proofs  because it takes advantage of the homological properties of the perfect ideal $I^*.$
     
     \vskip 2mm
     
We recall that  important   numerical information on a  perfect homogeneous ideal of codimension two  (in this case $I^*$) come from the Hilbert-Burch theorem:  all the possible Hilbert functions (more in general the graded Betti numbers) of an ideal minimally generated by $n$  forms with assigned degrees
$c_1, \dots, c_n$  are in one to one  correspondence with the $(n-1)$-tuples $ (e_2,\dots , e_n) $ such
that $c_i  < e_i <  c_{i+1}, $  $c_1+1 \ge n $ and $\sum_i c_i = \sum_i e_i. $ In our case more conditions will be necessary because $I^*$ is the leading ideal of a complete intersection of given valuations.  The theory of the cancellations come to our help. For our assumption  we need  a stronger condition given by item $(3) $ in the following result.  
\begin{theorem}\label{multiplicity}
Let $(S,\n)$ be a regular local ring of dimension 2 and let $I=(f,g)$ be a complete intersection ideal in $S$ with $v_{\n}(f)=d\leq v_{\n}(g)=b$. Let $\nu(I^*)=n$. Then there exist  two sequences  of integers $(c_1,\ldots, c_n)$ and $(0=e_1,e_2\ldots ,e_n)$ such that the following assertions hold true.
 \begin{enumerate}
 \item $c_1=d$ and $c_2=b$,
 \item $c_i+2\leq c_{i+1}$ for $i=2,\ldots ,n-1$,
 \item $c_i+1\leq e_i <c_{i+1}$ for $i=2,\ldots, n-1$ and $c_n+1\leq e_n$, 
 \item $\sum_{i=2}^n(e_i-c_i)=d$,
 \item $HS_{G}(t) =\frac{\sum_{i=1}^n(t^{e_i}-t^{c_i})}{(1-t)^2}$,
 \item $e(G)=\frac{\sum_{i=1}^n[e_i(e_i-1)-c_i(c_i-1)]}{2}$, where $e(G)$ denotes the multiplicity of $G$,
 \item $a(G)=e_n-2$, where $a(G)$ denotes the $a$-invariant of $G$.
 \end{enumerate}
\end{theorem}

\begin{proof}
Let $HF_{S/I}$ be the Hilbert function of $S/I$ and let $L$ be the lex-segment ideal corresponding to $HF_{S/I}$. The minimal free resolution of $L$ by \cite{EK} has the following shape:

$$0\longrightarrow \oplus_{j=1}^dP(-a_j-1)\longrightarrow P(-a_0)\oplus_{j=1}^dP(-a_j)\longrightarrow L\longrightarrow 0$$
where $a_0=d=v_{\n}(f)$ and $a_1=b=v_{\n}(g)$.
Note that in the resolution of $L$, $P(-a_0)$ and $P(-a_1)$ can not be cancelled since $a_0\leq a_1<a_j+1$ for $j=1,\ldots ,d$. We consider the two sets of integers corresponding to the shifts in homological positions $0$ and $1 $ in a minimal free resolution of $I^*:$   
$$0\longrightarrow \oplus_{j=2}^nP(-e_j)\longrightarrow P(-c_1)\oplus_{j=2}^nP(-c_j)\longrightarrow I^*\longrightarrow 0.$$

Then $c_1=a_0=d$ and $c_2=a_1=b$. Since the resolution of $I^*$ is obtained after performing all the possible zero cancellations we have $c_i+2\leq c_{i+1}$ for $i=2,\ldots ,n-1$. Since $\nu(I)=2$, by Rossi and Sharafan's result in \cite[Remark 4.7]{rs} we have to perform $n-2$ negative cancellations. Hence we have  $c_i+1\leq e_i <c_{i+1}$ for $i=2,\ldots, n-1$. Condition (4) is well known and it follows by the projective dimension two.

From the resolution of $I^*$ we can compute the Hilbert series and by   \cite[Lemma 4.1.13]{bh},  we have $HS_{G}(t)=\frac{\sum_{i=1}^n(t^{e_i}-t^{c_i})}{(1-t)^2}$ and hence $e(G)=\frac{\sum_{i=1}^n[e_i(e_i-1)-c_i(c_i-1)]}{2}$.

Since $\dm G=0$  the $a$-invariant of $G$ is  $a(G)=\max\{t|G_t\not=0\}=\max \{t|h_t\not= 0\} $. From the minimal free resolution of $L$ it is clear that $a(G)=a_d-1$. In the minimal free resolution of $L$, $P(-a_d-1)$ can never be cancelled since $a_j<a_d+1$ for $j=0,1,\ldots,d$ so it will remain in the minimal free resolution of $I^*$ and therefore $e_n=a_d+1$. Hence we have $a(G)=e_n-2$. 

\end{proof}

 In Theorem 2.2 we observe that item (2) follows from item (3), nevertheless we think that it is interesting to outline  that the degrees of the initial forms cannot be consecutive, except those of the  first generators.

\vskip 2mm

\begin{remark}\label{relation}
Goto-Heinzer-Kim in \cite[Theorem 1.3]{gkm1} have expressed the Hilbert series of $G$ in terms of the sequence $c_i$'s and $d_i$'s as follows:
$$HS_G(t)=\frac{\sum_{i=2}^nt^{d_i}(1-t^{d_{i-1}-d_i})(1-t^{c_i-d_i})}{(1-t)^2}$$ 
where $d_1=c_1>d_2 >\cdots >d_{n-1}>d_n=0.$ 
 
Now comparing this expression with Theorem \ref{multiplicity} (5), we obtain for $i=2,\ldots n$
  
 \begin{equation} \label{ed} e_i=c_i+(d_{i-1}-d_i)   \end{equation}  
 
\end{remark}

\vskip 2mm

Let notation be as in Theorem 2.2. In the next theorem we prove  that the   two sequences  of integers $c_i$'s and $e_i$'s are also sufficient for the existence of  a complete intersection ideal $I$ whose leading ideal has these invariants. The result is the analogous of    \cite[Theorem 2.3]{gkm2}.
\begin{theorem}\label{integer}
Let $a$ and $b$ be positive integers with $a\leq b$ and consider the following data:
\begin{enumerate}
\item An integer $n$ with $2\leq n\leq a+1$.
\item A sequence of integers $(a=c_1,b=c_2,c_3,\ldots ,c_n)$ such that $c_i+2\leq c_{i+1}$ for $i=2,\ldots ,n-1$.
\item A sequence of integers $(0=e_1,e_2,\ldots ,e_n)$ such that $c_{i}+1\leq e_i< c_{i+1}$ for $i=2,\ldots ,n-1$ and $c_n+1\leq e_n$ and $\sum_{i=2}^n(e_i-c_i)=a$.
\end{enumerate}
For each system satisfying these conditions there exists an ideal $I=(f,g)\subseteq S=k[|x,y|]$ generated by a complete intersection with $v_{\n}(f)=a$ and $v_{\n}(g)=b$ such that $\nu(I^*)=n$ and these two sequences of integers $c_i$'s and $e_i$'s completely determine the minimal free resolution of $I^*$ as follows:
 
$$0\longrightarrow \oplus_{j=2}^nP(-e_j)\longrightarrow P(-c_1)\oplus_{j=2}^nP(-c_j)\longrightarrow I^*\longrightarrow 0.$$
\end{theorem}

\begin{proof}
From the two sequence of integers $c_i$'s and $e_i$'s let us define a numerical function as follows:
\[
q(j)= \left\{
\begin{array}{rl}
 -1 &  \mbox { if } j=c_i \mbox{ for }3\leq i\leq n\\
     -2&   \mbox{ if } j=c_1=c_2\\
     -1 & \mbox{ if }  j=c_1,c_2 \mbox{  and } c_1\not=c_2\\
     1& \mbox{  if } j=e_i \mbox{  for } i=2,\ldots ,n\\
     0 &\mbox{ otherwise. }
     \end{array}
     \right.
\]
Then define another numerical function $p$ inductively as follows $$p(0)=1 \mbox{ and } p(j)=p(j-1)+q(j) \mbox{ for } j\geq 1.$$ Then we get 
\[
p(j)= \left\{
\begin{array}{rr}
 1 &  \mbox { if } 1\leq j<c_1\\
     -1&   \mbox{ if } c_i\leq j<e_i \mbox{ for } 2\leq i\leq n\\
     -1 & \mbox{ if }  j=c_1=c_2 \\
     0 &\mbox{ otherwise .}
     \end{array}
     \right.
\]

Now define the numerical function $h$ inductively as follows $$h(0)=1 \mbox{ and } h(j)=h(j-1)+p(j) \mbox{ for } j\geq 1.$$ Then we can observe that 
\begin{enumerate}
\item $h(j)=j+1$ for $0\leq j< c_1$ 
\item $h(j)\geq h(j+1)$ for $j\geq c_1$
\item $|h(j)-h(j-1)|=|p(j)|\leq 1$ for $j\geq 1$.
\end{enumerate}
Since $p(j)=-1$ for $c_i\leq j<e_i$ with $i=2,\ldots ,n$ and $\sum_{i=2}^n(e_i-c_i)=c_1$ we have $h(e_n-2)=1\geq h(e_n-1)=0=h(j)$ for all $j\geq e_n$. The above conditions imply that $h$ is an numerical function admissible for an Artinian graded $k$-algebra of codimension 2. Note that $p$ and $q$ actually denote the $\triangle h$ and $\triangle^2h$ respectively. Let $\mathcal I:=\{j|\triangle^2h(j)\leq -1\}$ and $\mathcal J:=\{j|\triangle^2h(j)\geq 1\}$. By \cite[Theorem 2.4]{b} there exists a complete intersection ideal $I=(f,g)$ such that $HF_{S/I}=h$. Then by Theorem \ref{unique} the graded Betti numbers of $I^*$ are uniquely determined by $h$ as follows
\[
\beta_{0j}(I^*) = \left\{
 \begin{array}{ll}
  |\triangle^2h(j) | &\mbox{ for }j\in \mathcal I\\
  0 &\mbox{ otherwise, }
      \end{array}
      \right.
\]
\[
      \beta_{1j}(I^*) = \left\{
       \begin{array}{ll}
        |\triangle^2h(j) | &\mbox{ for }j\in \mathcal J\\
        0 &\mbox{ otherwise. }
            \end{array}
            \right.
             \]
and the minimal free resolution of $I^*$ can be expressed in terms of these two sequences of integers as  follows:
$$0\longrightarrow \oplus_{j=2}^nP(-e_j)\longrightarrow P(-c_1)\oplus_{j=2}^nP(-c_j)\longrightarrow I^*\longrightarrow 0.$$
\end{proof}

We  give  a concrete example for clarifying the above result.

\begin{example}\label{example}
Consider the two sequences of integers $(4,5,8,11)$ and $(6,9,13)$  satisfying  the conditions of Theorem \ref{integer}. We will exhibit  a complete intersection ideal $I=(f,g) \subseteq k[[x,y]] $ such that the numerical invariants of the minimal free resolution of $I^*$ will be completely determined by these two sequence of integers, see equation (\ref{star}).  By Theorem \ref{multiplicity}(5) the Hilbert series corresponding to the two sequences $(4,5,8,11)$ and $(6,9,13)$ is given as
\begin{eqnarray*} HS_G(t)  &= & \frac{(1-t^{4}+t^{6}-t^{5}+t^{9}-t^{8}+t^{13}-t^{11})}{(1-t)^2}\\
&=&1+2t+3t^2+4t^3+4t^4+3t^5+3t^6+3t^7+2t^8+2t^9+2t^{10}+t^{11}
\end{eqnarray*}
 and the lex-segment ideal with the above Hilbert series is $L=(x^4,x^3y^2,x^2y^6,xy^{10},y^{12})$. The minimal free resolution of $L$ is of the form 
{\small $$0\longrightarrow P(-6)\oplus P(-9) \oplus P(-12)\oplus P(-13)\longrightarrow P(-4) \oplus P(-5) \oplus P(-8) \oplus P(-11) \oplus P(-12)\longrightarrow L$$}
 The resolution of $L$ admits one zero cancellation $(P(-12),P(-12))$. The resolution of $I^*$ is obtained by
 performing the zero cancellation as follows
\begin{equation} \label{star} 0\longrightarrow P(-6)\oplus P(-9) \oplus P(-13)\longrightarrow P(-4) \oplus P(-5) \oplus P(-8) \oplus P(-11) \longrightarrow I^* 
\end{equation} 
By  Bertella's construction as in Example  \ref{ex1}, then  $I^*=(-xy^{10}, x^2y^6 - y^8, -x^3y^2 + xy^4, x^4 - x^2y^2)$ is generated by the maximal minors of the following matrix 
\[
\begin{pmatrix}
y^2 & 0 &0 &0\\
-x & y^4 &0 &0\\
0& -x &y^4 &0\\
0 & 0 & -x & y^2\\
0 &0 & 1& -x 
\end{pmatrix}
\]
The resolution of $I^*$ admits two negative cancellations : $(P(-6),P(-8))$ and $(P(-9),P(-11))$. By performing the negative cancellation we obtain the minimal free resolution of $I$, where $I=(xy^6 - x^3y^2 + xy^4, -2x^2y^4 + y^6 + x^4 - x^2y^2)$ is generated by the maximal minors of the following matrix
\[
\begin{pmatrix}
y^2 & 0 &0 &0\\
-x & y^4 &0 &0\\
1& -x &y^4 &0\\
0 & 1 & -x & y^2\\
0 &0 & 1& -x 
\end{pmatrix}
\]

\end{example}

\vskip 2mm
In \cite[Theorem 1.6]{gkm1}, Goto, Heinzer and Kim extended their results \cite[Theorem 1.2 and 1.3]{gkm1} to dimension $s>2, $ provided $G$ is Cohen-Macaulay.   Using our approach  the extension is immediate  because the numerical invariants of the free resolution do not change modulo a regular sequence.

\begin{theorem}\label{cm}
Let $(S,\n)$ be a regular local ring of dimension $s > 2$. Let $I=(f,g)$ be an ideal generated by a regular sequence in $S$ and let  $R=S/I$ and $\m=\n/I$. We put $G:=gr_{\m}(R)$. If $G$ is Cohen-Macaulay ring and $\nu(I^*)=n, $ then there exists two sequences of integers 
$(c_1,\ldots,c_n)$ and $(e_1=0,e_2,\ldots,e_n)$ satisfying  the following conditions: 
\begin{enumerate}
\item $c_1=v_{\n}(f)$ and $c_2=v_{\n}(g)$
\item $c_i+2\leq c_{i+1}$ for $i=2,\ldots ,n-1$.
\item $c_i+1\leq e_i <c_{i+1}$ for $i=2,\ldots, n-1$ and $c_n+1\leq e_n$.  
\item $\sum_{i=2}^n(e_i-c_i)=v_{\n}(f)$.
\item $HS_G(t)=\frac{\sum_{i\geq 1}(t^{e_i}-t^{c_i})}{(1-t)^s}$.
\item $a(G)=e_n-s$.
\end{enumerate}
\end{theorem}

\begin{proof}
We may assume that the residue field $k$ is infinite. Let $\n=(x_1,x_2,\ldots ,x_s)$. Since $G$ is Cohen-Macaulay of dimension $s-2,$ after a change of coordinates  we may assume  that $Q=(x_3^*,\ldots ,x^*_s)$ is generated by a $G$-regular sequence. Let $ \q=(x_3,\ldots ,x_s)$ and consider  $\ov S=S/\q$ with maximal ideal $\ov \n=\n/\q$. Let $\ov I=(\ov f, \ov g)$ where $\ov f, \ov g$ denote the images of $f$ and $g$ in $\ov S$, respectively. Then $\ov I$ is a parameter ideal in $\ov S$ and hence $\ov R=\ov S/\ov I$ is the Artinian local ring of embedding dimension two with   maximal ideal $\ov {\m}=\ov {\n}/\ov I.$  Since $G$ is Cohen-Macaulay we have $\gr_{\ov {\m}}(\ov R)\simeq G/Q= \ov G. $ We put $P:=gr_{\n}(S)$. Now the following facts hold:

\vskip 2mm
- $\beta_{ij}^P  (G) = \beta_{ij}^{P/Q} (\ov G).$

\vskip 2mm
- $HS_G(t) =\frac{HS_{\ov G}(t)}{(1-t)^{s-2}}$

\vskip 2mm
- $a(G)=a(\ov G)-(s-2)$

\vskip 2mm
Thus by passing to $\ov R$ the result follows by Theorem \ref{multiplicity}.  

\end{proof}
  
\begin{remark}

From Theorem \ref{multiplicity}, it follows that the Hilbert series of $G$ is uniquely determined by two sequences of integers $(c_1,c_2,\ldots ,c_n)$ and $(e_1=0,e_2,\ldots ,e_n)$. Usually the sequence $(c_1,\ldots ,c_n) $ does not uniquely determine the Hilbert series of $G$. But if $\nu(I^*)=c_1+1, $ then the minimal free resolution of $P/I^*$ is same as the minimal free resolution of the lex segment ideal associated and in that case we know that $e_i=c_i+1$ for $=2,\ldots ,n$.  Hence the Hilbert series of $G$ is uniquely determined by the sequence $(c_1,c_2,\ldots ,c_n)$.

\end{remark}

Next example shows that in general the sequence $ (c_1,c_2,\ldots ,c_n)$ does not determine uniquely the Hilbert series of $G. $

\begin{example}
Let $(S,\n)$ be a regular local ring and $I$ be a complete intersection ideal. Given the admissible sequence $(c_1,c_2,c_3,c_4)=(4,5,8,11), $ then,  by Theorem \ref{integer},  the possible values  for $(e_2,e_3,e_4)$ are 
$$(6,9,13),(6,10,12) \mbox{ and }(7,9,12).$$
They correspond to    three different Hilbert series of  $G$ and each of them is realizable.\\
$(1)$  The two sequence of integers $(4,5,8,11)$ and $(6,9,13)$ have already been discussed in Example \ref{example}.\\ $(2)$  The Hilbert series corresponding to the two sequences $(4,5,8,11)$ and $(6,10,12)$ is given by 
\begin{eqnarray*}
HS_G(t)  &= &\frac{(1-t^{4}+t^{6}-t^{5}+t^{10}-t^{8}+t^{12}-t^{11})}{(1-t)^2}\\
 &=& 1+2t+3t^2+4t^3+4t^4+3t^5+3t^6+3t^7+2t^8+t^9+t^{10}
\end{eqnarray*}

and the corresponding lex-segment ideal is $L=(x^4,x^3y^2,x^2y^6,xy^{8},y^{11})$. The minimal free resolution of $L$ is of the form 
{\small $$0\rightarrow P(-6)\oplus P(-9)\oplus P(-10) \oplus P(-12)\longrightarrow P(-4) \oplus P(-5) \oplus P(-8) \oplus P(-9)\oplus P(-11) \longrightarrow L \rightarrow 0.$$}
The resolution of $L$ admits only one zero cancellation $(P(-9),P(-9))$. The resolution of $I^*$ is obtained by
 performing the zero cancellation as follows
$$0\longrightarrow P(-6)\oplus P(-10) \oplus P(-12)\longrightarrow P(-4) \oplus P(-5) \oplus P(-8) \oplus P(-11) \longrightarrow I^*\longrightarrow 0 .$$
For instance $I^*=(x^4 - x^2y^2, -x^3y^2 + xy^4, x^2y^6, y^{11})$ is generated by the maximal minors of the following matrix 
\[
\begin{pmatrix}
y^2 & 0 &0 &0\\
-x & y^4 &0 &0\\
0& -x &y^2 &0\\
0 & 1 & -x & y^3\\
0 &0 & 0& -x 
\end{pmatrix}
\]
The resolution of $I^*$ admits two negative cancellations : {\small{$(P(-6),P(-8))$ and $(P(-10),P(-11))$}}. By performing the negative cancellations  we obtain the minimal free resolution of $I$, where $I=(xy^5 - x^3y^2 + xy^4, y^7 - x^2y^4 - x^2y^3 + x^4 - x^2y^2)$ is generated by the maximal minors of the following matrix
\[
\begin{pmatrix}
y^2 & 0 &0 &0\\
-x & y^4 &0 &0\\
1& -x &y^2 &0\\
0 & 1 & -x & y^3\\
0 &0 & 1& -x 
\end{pmatrix}
\]

$(3)$ The Hilbert series corresponding to the two sequences $(4,5,8,11)$ and $(7,9,12)$ is given as

\begin{eqnarray*}
HS_G(t)  &=&\frac{(1-t^{4}+t^{7}-t^{5}+t^{9}-t^{8}+t^{12}-t^{11})}{(1-t)^2}\\
&=& 1+2t+3t^2+4t^3+4t^4+3t^5+2t^6+2t^7+t^8+t^9+t^{10}
\end{eqnarray*}

and the  lex segment ideal with the above Hilbert function is $L=(x^4,x^3y^2,x^2y^4,xy^{7},y^{11})$. By repeating the above procedure we obtain $I^*=(x^4 - x^2y^2, -x^3y^2, -xy^7, y^{11})$ and $I=(xy^6 + xy^5 - x^3y^2, -x^2y^4 + y^6 - x^2y^3 + x^4 - x^2y^2). $

\end{example}

\medskip
 
\section{When $\nu(I)=3$}
  In this section we give a more precise upper bound on the  number of generators of $I^*, $ provided $\nu(I)=3$. It is clear that the method used in the above section suggests a procedure for getting information on the Betti table of $I^*$ anyway  the number of generators of $I$ will be  fixed. The problem is that the graph of the possible cancellations is too complicate for a general description. 
Let $h$ be an $O$-sequence and let us define $$\mathcal H=\{i|\triangle^2h(i)=0, \triangle h(i)=-1 \},$$  
$$\mathcal I=\{j|\triangle^2h(j)\leq -1\}\ \ \ \  \mbox{and}  \ \ \  \mathcal J=\{j|\triangle^2h(j)\geq 1\}.$$
 
\begin{theorem}\label{three}
Let $h:=(1,2,\ldots, d,h_d,\ldots ,h_s)$ be an $O$-sequence and 
let $I $  be an ideal of height two of a $2$-dimensional regular local ring $S$  such that   $HF_{S/I}=h.  $ If $\nu(I)=3, $  then $$\displaystyle{\sum_{\triangle^2h(i)\leq -1}|\triangle^2h(i)|}\leq \nu(I^*)\leq \displaystyle{\sum_{\triangle^2h(i)\leq -1}|\triangle^2h(i)|}+ \displaystyle{\sum_{i\geq 0}\delta_i(\mathcal H)}$$
 where $ \delta_i(\mathcal H)=\left\{
  \begin{array}{ll}
   1 &\mbox{ if }i\in \mathcal H\\
   0 &\mbox{ otherwise .}
       \end{array}
       \right.
 $
\end{theorem}

\begin{proof}
Notice that $I^*$ is a homogeneous perfect ideal of height two in $P:=gr_{\n}(S)$. By Theorem 1.1, we have $\sum_{i\in \mathcal I}|\triangle^2h(i)|\leq \nu(I^*)$. Since $\nu(I)=3$ then by \cite[Theorem 2.3]{b} we have $p=\max_{i\geq 1}\{|\triangle h(i)|\}\leq 2$. Let $L$ be the lex-segment ideal corresponding to the numerical function $h$. Let us denote by $p_j:=\triangle h(j)$. Since $L$ is a Borel fixed ideal by \cite{EK} the  minimal free resolution of $L$  is of the form

{$$0\longrightarrow \bigoplus_{j=1}^t P^{|p_{c_j}|}(-c_j-1)\longrightarrow P(-d)\bigoplus_{j=1}^tP^{|p_{c_j}|}(-c_j)\longrightarrow L\longrightarrow 0$$}where $d\leq c_1<c_2<\ldots <c_t$. By \cite[Theorem 4.1]{rs} the Betti numbers of $I$  come from the Betti numbers of $L$ 
by performing a sequence of consecutive zero and negative cancellations and the minimal free resolution of $I$ is of the form $$0\longrightarrow F_2=S^2\longrightarrow F_1=S^3\longrightarrow I\longrightarrow 0.$$
  By \cite[Theorem 1.1]{p} the Betti numbers of $I^*$ come from the Betti numbers of  $L$ by a sequence of consecutive zero cancellation. Thus in order to have the upper bound of the number of generators of $I^*$ we need to find how many zero cancellations in the minimal free resolution of $L$ can be replaced by negative cancellations. Let us assume that we have a zero cancellation  at the $c_j$-th place.  Then we have $c_{j-1}+1=c_j$. There are three possibilities for $\triangle^2h(c_j)$ namely $-1,1$ or $0$ as $|p_{c_j}|\leq 2$ for all $j=1,\ldots t$. 
First assume that we have  $\triangle^2h(c_j)=1$ then the resolution will be of the form
$$\longrightarrow \cdots P^2(-c_j)\oplus P(-c_j-1)\oplus \cdots \longrightarrow \cdots P^2(-c_{j-1})\oplus P(-c_j)\oplus P^{|p_{c_{j+1}}|}P(-c_{j+1})\cdots$$
If we don't perform the zero cancellation $(P(-c_j),P(-c_j))$ then $\sum_{i=j-1}^t\beta_{1c_i+1}(L)=\sum_{i=j+1}^t\beta_{0c_i}(L)+3$ which forces rank$F_2= 3$ which is a contradiction. Thus the zero cancellation at the $c_j$-th place can not be replaced by negative cancellation.

Next assume that we have a zero cancellation at the $c_j$-th place with $\triangle^2h(c_j)=-1$ then the resolution will be of the form 
$$\longrightarrow \cdots P(-c_j)\oplus P^2(-c_j-1)\oplus \cdots \longrightarrow \cdots P(-c_{j-1})\oplus P^2(-c_j)\oplus P^{|p_{c_{j+1}}|}P(-c_{j+1})\cdots$$
If we don't perform the zero cancellation $(P(-c_j),P(-c_j))$ then  $\sum_{i=j-1}^t\beta_{1c_i+1}(L)=\sum_{i=j+1}^t\beta_{0c_i}(L)+3$ which forces rank$F_2= 3$ which is a contradiction. Thus the zero cancellation at the $c_j$-th place can not be replaced by negative cancellation.

Let us assume that we have a zero cancellation at the $c_j$-th place and $\triangle^2 h(c_j)=0$. It can happen in the following two ways. First consider the following case:
$$\longrightarrow \cdots  P^2(-c_j)\oplus P^2(-c_j-1)\oplus \cdots \longrightarrow \cdots P^2(-c_{j-1})\oplus P^2(-c_j)\oplus P^{|p_{c_{j+1}}|}P(-c_{j+1})\cdots$$ where $\triangle h(c_j)=-2$.
If we don't perform both the zero cancellation $(P^2(-c_j),P^2(-c_j))$ then $\sum_{i=j-1}^t\beta_{1c_i+1}(L)=\sum_{i=j+1}^t\beta_{0c_i}(L)+4$ which forces rank$F_2= 4$ which is a contradiction. Thus the zero cancellation at the $c_j$-th place can not be replaced by negative cancellation.
Next consider the following case 
$$\longrightarrow \cdots P(-c_j)\oplus P(-c_j-1)\oplus \cdots \longrightarrow \cdots P(-c_{j-1})\oplus P(-c_j)\oplus P^{|p_{c_{j+1}}|}P(-c_{j+1})\cdots$$ where $\triangle h(c_j)=-1$
If we don't perform  the zero cancellation $(P(-c_j),P(-c_j))$ then $\sum_{i=j-1}^t\beta_{1c_i+1}(L)=\sum_{i=j+1}^t\beta_{0c_i}(L)+2$ which forces rank$F_2= 2$. Thus the zero cancellation $(P(-c_j),P(-c_j))$ can be replaced by a negative cancellation and it will increase the number of generators of $I^*$. Therefore we have 
$$\nu(I^*)\leq \displaystyle{\sum_{\triangle^2h(i)\leq -1}|\triangle^2h(i)|}+\displaystyle{\sum_{i\geq 0}\delta_i(\mathcal H)}$$
\end{proof}

Goto-Heinzer-Kim in \cite[Corollary 1.3]{gkh} proved that if   $I$ is a complete intersection ideal of codimension two,  then $I^*$ is a perfect ideal provided  $\nu(I^*) = 3.$ Inspired by this result we  may ask  if, given    a perfect ideal $I$ of codimension two,   the condition  $\nu(I^*) = \nu(I) +1 $   could be  sufficient for having $I^*$ perfect. At the moment the answer is unknown. The authors thank Youngsu Kim for pointing out   a mistake in a related example which appears in a preliminary version of this  paper.


\begin{thebibliography}{AAAA}

\bibitem{b} V. Bertella
{\em Hilbert function of local Artinian level rings in codimension 2},
Journal of Algebra, {\bf 321} (2009), 1429--1442.

\bibitem{BrIar}  J. Briancon, A. Iarrobino,
{\em Dimension of the punctual Hilbert Scheme},
Journal of Algebra, {\bf 55} (1978), 536--544.

\bibitem {bh} W. Bruns and J. Herzog, Cohen-Macaulay Rings, Revised Edition, 
Cambridge University Press, 1998.

\bibitem {CoC} CoCoA Team, CoCoA: a system for doing Computations in Commutative Algebra, Avail- able at http://cocoa.dima.unige.it.

\bibitem{EK} S. Eliahou and M. Kervaire, 
{\em Minimal resolutions of some monomial ideals},
 Journal of Algebra {\bf 129} (1990),  1-–25.
 
 \bibitem{Elias} J. Elias,
{\em A sharp bound for the minimal number of generators of perfect height two ideals},
 Manuscripta Math.  {\bf 55} (1986),  93--99.
 
  \bibitem{ERV} J. Elias, L. Robbiano and G. Valla,
{\em Number of generators of  ideals},
 Nagoya Math. J. {\bf 123} (1991),  39--76.

\bibitem {gkh}S. Goto, W. Heinzer and M. Kim,
{\em The leading ideal of a complete intersection of height two},
Journal of Algebra, {\bf 298} (2006), 238–-247.

\bibitem {gkm1}S. Goto, W. Heinzer and M. Kim,
{\em The leading ideal of a complete intersection of height two, Part II},
Journal of Algebra, {\bf 312} (2007), 709–-732.

\bibitem {gkm2}S. Goto, W. Heinzer and M. Kim,
{\em The leading ideal of a complete intersection of height two in a 2-dimensional regular local ring},
Communications in Algebra, {\bf 36} (2008), 1901--1910.

\bibitem {I}  A. Iarrobino, {\em Punctual Hilbert schemes}, Mem. Amer. Math. Soc. {\bf 188}(1977).

\bibitem {k} S.C.  Kothari, {\em The local Hilbert function of a pair of plane curves}, Proc. Amer. Math. Soc., {\bf 72(3)} (1978), 439–-442.

\bibitem{p} I. Peeva,
{\em Consecutive cancellations in betti numbers},
Proc. Amer. Math. Soc., {\bf 132} (2004), 3503--3507.

\bibitem {rs}M. E. Rossi and L. Sharifan,
{\em Consecutive cancellations in Betti numbers of local rings},
Proc. Amer. Math. Soc., {\bf 138}(1) (2010), 61–-73.

\bibitem {Sally1} J. D. Sally,
{\em Bounds for number of generators of Cohen-Macaulay ideals },
Pac. J.  Math. , {\bf 63}  (1976), 517--520.

\bibitem {Sally2} J. D. Sally,
{\em Number of generators of ideals in local rings },
Lecture Notes in Pure and Applied Mathematics , {\bf 35} New York, Marcel Dekker  (1978).
\end{thebibliography}
\end{document}